\documentclass[12pt,a4paper]{article}
\usepackage[a4paper]{geometry}

\usepackage{tikz}

\usepackage{eucal}

\usepackage[all]{xy}

\usepackage{amssymb, eucal, amsmath, amsthm}

\usepackage{mathptmx}
\usepackage[scaled=.90]{helvet}
\usepackage{courier}

\usepackage[pdftex]{hyperref}

\hypersetup{
  colorlinks = true,
  linkcolor = black,
  urlcolor = blue,
  citecolor = black,
}

\renewcommand{\phi}{\varphi}

\newcommand{\op}{\mathrm{op}}

\newcommand{\bfB}{\mathbf{B}}
\newcommand{\bfC}{\mathbf{C}}
\newcommand{\bfD}{\mathbf{D}}
\newcommand{\bfF}{\mathbf{F}}
\newcommand{\bfG}{\mathbf{G}}
\newcommand{\bfM}{\mathbf{M}}

\newcommand{\bfZ}{\mathbf{Z}}

\newcommand{\bbB}{\mathbb{B}}
\newcommand{\bbC}{\mathbb{C}}
\newcommand{\bbD}{\mathbb{D}}
\newcommand{\bbG}{\mathbb{G}}
\newcommand{\bbM}{\mathbb{M}}
\newcommand{\bbN}{\mathbb{N}}
\newcommand{\bbT}{\mathbb{T}}
\newcommand{\bbZ}{\mathbb{Z}}

\newcommand{\rmA}{\mathrm{A}}
\newcommand{\rmB}{\mathrm{B}}
\newcommand{\rmC}{\mathrm{C}}
\newcommand{\rmE}{\mathrm{E}}
\newcommand{\rmH}{\mathrm{H}}
\newcommand{\rmL}{\mathrm{L}}
\newcommand{\rmM}{\mathrm{M}}
\newcommand{\rmR}{\mathrm{R}}
\newcommand{\rmS}{\mathrm{S}}
\newcommand{\rmU}{\mathrm{U}}
\newcommand{\rmZ}{\mathrm{Z}}

\newcommand{\rmd}{\mathrm{d}}

\newcommand{\frK}{\mathfrak{K}}

\newcommand{\bfVec}{\mathbf{Vec}}
\newcommand{\bfHo}{\mathbf{Ho}}
\newcommand{\bfFun}{\mathbf{Fun}}
\newcommand{\bfMor}{\mathbf{Mor}}
\newcommand{\rmMor}{\mathrm{Mor}}
\newcommand{\bfEnd}{\mathbf{End}}

\DeclareMathOperator{\id}{id}
\DeclareMathOperator{\Id}{Id}

\DeclareMathOperator{\Hex}{Hex}
\DeclareMathOperator{\Rep}{Rep}

\DeclareMathOperator{\GL}{GL}
\DeclareMathOperator{\Aut}{Aut}
\DeclareMathOperator{\Iso}{Iso}
\DeclareMathOperator{\Ker}{Ker}
\DeclareMathOperator{\Map}{Map}

\newtheorem{theorem}{Theorem}
\newtheorem{proposition}[theorem]{Proposition}
\newtheorem{corollary}[theorem]{Corollary}

\theoremstyle{definition}
\newtheorem{definition}[theorem]{Definition}
\newtheorem{example}[theorem]{Example}

\newtheorem{remark}[theorem]{Remark}

\numberwithin{theorem}{section}
\numberwithin{equation}{section}

\newcommand{\ot}{\otimes}
\newcommand{\del}{\partial}

\title{Drinfeld centers for bicategories}

\author{Ehud Meir and Markus Szymik}

\date{December 2014}

\begin{document}

\maketitle

\renewcommand{\abstractname}{}

\begin{abstract}%
\noindent
We generalize Drinfeld's notion of the center of a tensor category to bicategories. In this generality, we present a spectral sequence to compute the basic invariants of Drinfeld centers: the abelian monoid of isomorphism classes of objects, and the abelian automorphism group of its identity object. There is an associated obstruction theory that explains the difference between the Drinfeld center and the center of the classifying category. For examples, we discuss bicategories of groups and bands, rings and bimodules, as well as fusion categories. 

\vspace{\baselineskip}
\noindent MSC: 
18D05 (primary),  	
55T99 (secondary).	

\vspace{\baselineskip}
\noindent Keywords: Drinfeld centers, bicategories, spectral sequences, obstruction theory, bands, bimodules, fusion categories.
\end{abstract}


\section*{Introduction}

If~$M$ is a monoid, then its center~$\rmZ(M)$ is the abelian submonoid of elements that commute with all elements. 
Monoids are just the (small) categories with only one object. It has therefore been natural to ask for a generalization of the center construction to categories~$\bfC$. 
The resulting notion is often referred to as the {\em Bernstein center}~$\rmZ(\bfC)$ of~$\bfC$, see~\cite[II~\S2]{Bass}, \cite[II.5, Exercise~8]{MacLane} as well as~\cite[1.9]{Bernstein}, for example. It is the abelian monoid of natural transformations~$\Id(\bfC)\to\Id(\bfC)$, 
so that its elements are the families~$(p(x)\colon x\to x\,|\,x\in\bfC)$ of self-maps that commute with all morphisms in~$\bfC$. 
Centers in this generality have applications far beyond those provided by monoids alone. As an example which is at least as important as elementary, for every prime number~$p$ the center of the category of commutative rings in characteristic~$p$ is freely generated by Frobenius, exhibiting Frobenius as a universal symmetry of commutative algebra in prime characteristic.

Drinfeld, Joyal, Majid, Street, and possibly others have generalized the notion of center in a different direction, 
from monoids to (small) monoidal (or tensor) categories, 
see~\cite[Definition~3]{Joyal+Street:tortile} 
and~\cite[Example~3.4]{Majid}. 
The resulting notion is often referred to as the {\em Drinfeld center}. Since tensor categories are just the bicategories with only one object, it is therefore natural to ask for a generalization of the Drinfeld center construction to bicategories. This will be the first achievement in the present paper. 

After we have set up our conventions and notation for bicategories in Section~\ref{sec:review}, Section~\ref{sec:refs_and_props} contains our definition and the main properties of the Drinfeld center of bicategories: The center is a braided tensor category that is invariant under equivalences. In the central Section~\ref{sec:ss}, we will explain a systematic method (a spectral sequence) to compute the two primary invariants of the Drinfeld center of every bicategory as a braided tensor category: the abelian monoid of isomorphism classes of objects, and the abelian group of automorphisms of its unit object. 

We will also explain the relation of the Drinfeld center with a more primitive construction: the center of the classifying category. These are connected by a characteristic homomorphism~\eqref{eq:characteristic} that, in general, need not be either injective or surjective. As an explanation of this phenomenon, we will see that the characteristic homomorphism can be interpreted as a fringe homomorphism of our spectral sequence. The word~`fringe' here refers to the fact that spectral sequences in non-linear contexts only rarely have a well-defined edge. As in~\cite{Bousfield}, this will lead us into an associated obstruction theory that will also be explained in detail. 

The final Sections~\ref{sec:groups},~\ref{sec:fusion}, and~\ref{sec:bimodules} discuss important examples where our spectral sequence can be computed and where it sheds light on less systematic approaches to computations of centers: the~$2$-category of groups, where the~$2$-morphisms are given by conjugations, fusion categories in the sense of~\cite{ENO}, where Drinfeld centers have been in the focus from the beginning of the theory on, and the bicategory that underlies Morita theory: the bicategory of rings and bimodules .


\section{A review of bicategories}\label{sec:review}

To fix notation, we review the definitions and some basic examples of~$2$-categories and bicategories in this section. 
Some useful references for this material are \cite{Benabou}, \cite{Kelly+Street}, \cite{Kapranov+Voevodsky}, \cite{Street}, \cite{Lack}, and of course~\cite[Chapter~XII]{MacLane}. 
For clarity of exposition, we will use different font faces for ordinary categories and~$2$-categories/bicategories.

Ordinary categories will be denoted by boldface letters such as~$\bfC,\bfD,\dots$. and their objects will be denoted by~$x,y,\dots$. 
The set of morphisms in~$\bfC$ from~$x$ to~$y$ will be denoted by~$\rmMor_\bfC(x,y)$, or sometimes~$\bfC(x,y)$ for short. We will write~$x\in\bfC$ to indicate that~$x$ is an object of~$\bfC$. 
If~$\bfC$ is a small ordinary category, then~$\Iso(\bfC)$ will be the set of isomorphism classes of objects. If~$x\in\bfC$ is any object, then~$\Aut_\bfC(x)$ 
will be its automorphism group in~$\bfC$.

\begin{definition} A~{\em$2$-category} (or sometimes: {\em strict~$2$-category}) is a category enriched in small categories. This is a category~$\bbB$ in which for every two objects~(or~$0$-morphisms)~$x,y$ in~$\bbB$ the morphism set is the underlying object set of a given small category~$\bfMor_{\bbB}(x,y)$; there are identity objects, and a composition functor that satisfy the evident axioms.
\end{definition}

\begin{example}
A basic example of a~(large)~$2$-category has a objects the (small) categories, and the categories of morphisms are the functor categories~$\bfFun(\bfC,\bfD)$ 
with natural transformations as morphisms. We will write~$\bfEnd(\bfC)=\bfFun(\bfC,\bfC)$ for short.
\end{example}

Bicategories (or sometimes: lax/weak~$2$-categories) are similar to~$2$-categories, 
except that the associativity and identity properties are not given by equalities, but by natural isomorphisms.
More precisely, we have the following definition.

\begin{definition}
A {\em bicategory}~$\bbB$ consists of
\begin{itemize}
\item objects (the~{\em$0$-morphisms})~$x,y,...$,
\item a category~$\bfMor_{\bbB}(x,y)$ of morphisms (the~{\em$1$-morphisms}) for every ordered pair of objects~$x,y$ of~$\bbB$,
\item a functor, the {\em horizontal composition},
\begin{align*}
\bfMor_{\bbB}(y,z)\times \bfMor_{\bbB}(x,y)&\longrightarrow\bfMor_{\bbB}(x,z),\\
(M,N)&\longmapsto M\ot N,
\end{align*}
\item and an identity object~$\Id(x)\in \bfMor_{\bbB}(x,x)$ for every object~$x$.
\end{itemize}
We also require natural transformations~$\alpha$,~$\lambda$ and~$\rho$ of functors that make the canonical associativity and identity diagrams commute.
\end{definition}

\begin{example}
A basic example of a~(large)~bicategory has as objects the (small) categories, and the categories of morphisms are the bimodules (or pro-functors or distributors), 
see \cite[Proposition~7.8.2]{Borceux1}.
\end{example}

Bicategories will be denoted by blackboard bold letters such as~$\bbB,\bbC,\dots$. The objects~(that is, the~$0$-morphisms) of a bicategory will be denoted by small letters~$x,y,...$. Objects in~$\bfMor_{\bbB}(x,y)$ (that is, the~$1$-morphisms in~$\bbB$) will be denoted by capital letters~$M,N,\ldots$, and we will sometimes write~\hbox{$\bbB(x,y)=\bfMor_{\bbB}(x,y)$} for short.

Every category of the form~$\bfMor_\bbB(x,x)$ is a tensor category with respect to~$\otimes$ and~$\Id(x)$:

\begin{example}\label{ex:tensor_categories}
We refer to \cite[Chapter~VII]{MacLane} and \cite{Joyal+Street:geometry} for introductions to tensor~(or monoidal) categories. 
Every tensor category~$\bfM$ defines a bicategory~$\bbB(\bfM)$ with one just object, which will be denoted by~$\star$:
\[
\bfMor_{\bbB(\bfM)}(\star,\star)=\bfM.
\] 
Conversely, every bicategory with precisely one object is of this form. 
Examples of the form~$\bfEnd(\bfC)=\bfFun(\bfC,\bfC)$ for small categories~$\bfC$ should be thought of as typical in the sense that the product need not be symmetric. 
Weaker notions, such as braidings~\cite{Joyal+Street:braided}, will be described later when needed.
\end{example}

\begin{example}
On the one hand, every category~$\bfC$ defines a (`discrete') bicategory~$\bbD(\bfC)$, 
where the morphism sets~$\rmMor_\bfC(x,y)$ from~$\bfC$ are interpreted as categories~$\bfMor_{\bbD(\bfC)}(x,y)$ with only identity arrows. 
These examples are in fact always~$2$-categories.
\end{example} 

\begin{definition}
Every bicategory~$\bbB$ determines an ordinary category~$\bfHo(\bbB)$, its {\em classifying category}, as follows~\cite[Section~7]{Benabou}: The objects of~$\bfHo(\bbB)$ and~$\bbB$ are the same; the morphism set from~$x$ to~$y$ in the classifying category~$\bfHo(\bbB)$ is the set of isomorphism classes of objects in the morphism category~$\bfMor_\bbB(x,y)$. In the notation introduced before,
\[
\rmMor_{\bfHo(\bbB)}(x,y)=\Iso(\bfMor_\bbB(x,y)).
\]
The associativity and identity constraints for~$\bbB$ prove that horizontal composition provides~$\bfHo(\bbB)$ with a the structure of an ordinary category such that the isomorphism classes of the~$\Id(x)$ become the identities.
\end{definition}

\begin{remark}
The classifying category of a bicategory is not to be confused with the {\em Poincar\'e category} of a bicategory, see~\cite[Section~7]{Benabou} again.
\end{remark}

\begin{example}
If~$\bfM$ is a tensor category, and~$\bbB=\bbB(\bfM)$ is the associated bicategory with one object, then~$\bfHo(\bbB)$ is the monoid~$\Iso(\bfM)$ of isomorphism classes of objects in~$\bfM$, thought of as an ordinary category with one object.
\end{example}

\begin{example}
If~$\bbB=\bbD(\bfC)$ is a discrete bicategory defined by an ordinary category~$\bfC$, then the classifying category~$\bfHo(\bbB)=\bfC$ gives back the ordinary category~$\bfC$.
\end{example}

\begin{example}
There is a bicategory~$\bbT$ that has as objects the topological spaces, as~$1$-morphisms the continuous maps, and as~$2$-morphisms the homotopy classes of homotopies between them. In other words, one may think of~$\bfMor_\bbT(X,Y)$ as the fundamental groupoid of the space of maps~$X\to Y$ (with respect to a suitable topology). The bicategory~$\bbT$ is actually a~$2$-category. Its classifying category~$\bfHo(\bbT)$ is the homotopy category of topological spaces with respect to the (strong) homotopy equivalences. 
\end{example}

The preceding example explains our choice of notation~$\bfHo(\bbB)$ for general bicategories~$\bbB$.

We will not recall the appropriate notions of functors and natural transformations for bicategories here, but we note that every bicategory is equivalent to a~$2$-category, 
see~\cite{MacLane+Pare}.


\section{Drinfeld centers for bicategories}\label{sec:refs_and_props}

Drinfeld centers for tensor categories were introduced independently by Drinfeld, Majid~\cite[Example~3.4]{Majid}, and Joyal-Street~\cite[Definition~3]{Joyal+Street:tortile}. 
In this section, we extend their definition and basic properties to bicategories.

\subsection{Definition}

Drinfeld centers for bicategories are defined as follows.

\begin{definition}\label{def:Drinfeld_center}
Let~$\bbB$ be a (small) bicategory. Its {\em Drinfeld center}~$\bfZ(\bbB)$ is the following ordinary category. The objects in~$\bfZ(\bbB)$ are pairs~$(P,p)$ where
\[
P=(P(x)\in\bfMor_\bbB(x,x)\,|\,x\in\bbB)
\]
is a family of objects in the endomorphism (tensor) categories~$\bfMor_\bbB(x,x)$, one for each object~$x\in\bbB$, and
\[
p=(p(M)\colon P(y)\otimes M\overset{\cong}{\longrightarrow} M\otimes P(x)\,|\,x,y\in\bbB, M\in\bfMor_\bbB(x,y))
\] 
is a family of natural isomorphisms in~$\bfMor_\bbB(x,y)$, one for each object~$M\in\bfMor_\bbB(x,y)$.
These pairs of families have to satisfy two conditions: 
Firstly, the isomorphism
\[
p(\Id(x))\colon P(x)\otimes \Id(x)\to\Id(x)\otimes P(x)
\]
is the composition of the identity constraints~$\lambda$ and~$\rho$. Secondly, ignoring the obvious associativity constraints, there is an equality 
\begin{equation}\label{eq:triangle}
p(M\otimes N)=(\id(M)\otimes p(N))(p(M)\otimes\id(N))
\end{equation}
that holds between morphisms in the category~$\bfMor_\bbB(x,z)$ for all objects~$M\in\bfMor_\bbB(y,z)$ and~\hbox{$N\in\bfMor_\bbB(x,y)$} that are horizontally composeable.

A {\em morphism} from~$(P,p)$ to~$(Q,q)$ in the category~$\bfZ(\bbB)$ is a family of morphisms
\[
f(x)\colon P(x)\to Q(x)
\]
in the categories~$\bfMor_\bbB(x,x)$ such that the diagram
\[
\xymatrix{
P(y)\otimes M\ar[d]_{f(y)\otimes\id(M)}\ar[r]^{p(M)}&M\otimes P(x)\ar[d]^{\id(M)\otimes f(x)}\\
Q(y)\otimes M\ar[r]_{q(M)}&M\otimes Q(x)
}
\]
commutes for all objects~$M\in\bfMor_\bbB(x,y)$. Identities and composition in the category~$\bfZ(\bbB)$ are defined so that there is a faithful functor
\[
\bfZ(\bbB)\longrightarrow\prod_{x\in\bbB}\bfMor_\bbB(x,x)
\]
to the product of the endomorphism categories which is forgetful on objects.
\end{definition}

\begin{remark}\label{rem:higher}
Equation~\eqref{eq:triangle} means that the triangles 
\[
\xymatrix@R=30pt@C=10pt{
&M\otimes P(y)\otimes N\ar[dr]^{\id(M)\otimes p(N)}&\\
P(z)\otimes M\otimes N\ar[ur]^{p(M)\otimes\id(N)}\ar[rr]_{p(M\otimes N)}&&M\otimes N\otimes P(x)
}
\]
in the category~$\bfMor_\bbB(x,z)$ commute on the nose, again ignoring the given associativity constraints. 
This means that the bottom arrow is {\em equal} to the composition of the other two. For bicategories as defined here, 
there is no other notion of {\em equivalence} between morphisms in~$\bfMor_\bbB(x,z)$, so this is the appropriate definition in our situation. 
Further generalization--with even higher order structure~$(P^0,P^1,P^2,\dots)$ instead of just~$(P,p)$--is called for in higher categories. This has been developed for the context of simplicial categories in~\cite{Szymik}. A detailed comparison, while clearly desirable, is not within the scope of the present text. We focus entirely on the categorical situation here.
\end{remark}

\begin{example}
By inspection, if~$\bbB=\bbB(\bfM)$ is a tensor category, thought of as a bicategory with one object as in Example~\ref{ex:tensor_categories}, 
then our definition recovers the Drinfeld center as defined in~\cite[Example~3.4]{Majid} and~\cite[Definition~3]{Joyal+Street:tortile}.
\end{example}

Other examples related to categories of groups, bands, and fusion categories will be discussed later, see Section~\ref{sec:groups} and Section~\ref{sec:fusion}, respectively.


\subsection{Basic properties}

We now list the most basic properties of Drinfeld centers for bicategories: They are invariant under equivalences, 
and carry canonical structures of braided tensor categories. All of these are straightforward generalizations from the one-object case of tensor categories.

\begin{proposition}
For every (small) bicategory~$\bbB$, its Drinfeld center~$\bfZ(\bbB)$ has a canonical structure of a tensor category. The tensor product
\[
(P,p)\otimes(Q,q)=(P\otimes Q,p\otimes q)
\] 
is defined by
\[
(P\otimes Q)(x)=P(x)\otimes Q(x)
\]
and
\[
(p\otimes q)(M)=(p(M)\otimes\id(Q(x)))(\id(P(y))\otimes q(M))
\]
as morphisms
\[
P(y)\otimes Q(y)\otimes M\longrightarrow
P(y)\otimes M\otimes Q(x)\longrightarrow
M\otimes P(x)\otimes Q(x)
\]
for~$M\in\bfMor_\bbB(x,y)$.
The tensor unit is~$(E,e)$ with
\[
E(x)=\Id(x)
\]
and
\[
e(M)\colon\Id(y)\otimes M\cong M\cong M\otimes\Id(x)
\]
is given by the constraints of the tensor structure.
\end{proposition}

\begin{proof}
See~\cite[XIII.4.2]{Kassel} for the case of tensor categories.
\end{proof}

\begin{corollary}
The group~$\Aut_{\bfZ(\bbB)}(E,e)$ is abelian.
\end{corollary}

\begin{proof}
In every (small) tensor category, the endomorphism monoid of the tensor unit is abelian. See~\cite[XI.2.4]{Kassel}, for example.
\end{proof}

\begin{proposition}
For every (small) bicategory~$\bbB$, its Drinfeld center~$\bfZ(\bbB)$ has a canonical structure of a braided tensor category. The braiding
\[
(P,p)\otimes(Q,q)\longrightarrow(Q,q)\otimes(P,p)
\]
is defined by the morphisms
\[
p(Q(x))\colon P(x)\otimes Q(x)\longrightarrow Q(x)\otimes P(x)
\] 
for objects~$x\in\bbB$.
\end{proposition}

\begin{proof}
See again~\cite[XIII.4.2]{Kassel} or~\cite[Proposition~4]{Joyal+Street:tortile} for the case of tensor categories.
\end{proof}

\begin{corollary}
The monoid~$\Iso(\bfZ(\bbB))$ is abelian.
\end{corollary}

\begin{proof}
In every (small) braided tensor category, the monoid of isomorphism classes of objects is abelian. 
\end{proof}

\begin{remark}
The Drinfeld center is invariant under equivalences of bicategories. This is shown for tensor categories in~\cite{Mueger}, even under the more general hypothesis that the tensor categories are weakly monoidal Morita equivalent. This will not be needed in the following.
\end{remark}



\section{Symmetries, deformations, and obstructions}\label{sec:ss}

In this section, we will explain a systematic method how to compute the two primary invariants of the Drinfeld center~$\bfZ(\bbB)$ of every bicategory~$\bbB$ as a braided tensor category: 
The abelian monoid~$\Iso(\bfZ(\bbB))$ of isomorphism classes of objects under~$\otimes$, and the abelian group~$\Aut_{\bfZ(\bbB)}(E,e)$ of automorphisms of its unit of object. We will also explain the relation of the Drinfeld center to a more primitive construction: the center of the classifying category. These are connected by the homomorphism
\begin{equation}\label{eq:characteristic}
\Iso(\bfZ(\bbB))\longrightarrow\rmZ(\bfHo(\bbB))
\end{equation}
that sends the isomorphism class of an object~$(P,p)$ with~$P=(P(x)\in\bfMor_\bbB(x,x)\,|\,x\in\bbB)$ to the family~$[\,P\,]=([\,P(x)\,]\in\Iso\bfMor_\bbB(x,x)\,|\,x\in\bbB)$ of isomorphism classes.

\begin{definition}
The canonical homomorphism \eqref{eq:characteristic} is called the {\em characteristic homomorphsim}.
\end{definition}

In general, the characteristic homomorphism need not be either injective or surjective. As an explanation of this phenomenon, we will see that the characteristic homomorphism can be interpreted as a fringe homomorphism of a spectral sequence~$(\rmE_r^{s,t}\,|\,r\geqslant1)$ with an associated obstruction theory. This spectral sequence will compute the abelian monoid~~$\Iso(\bfZ(\bbB))$ from~$\rmE_\infty^{s,t}$ with~$t-s=0$ and the abelian group~$\Aut_{\bfZ(\bbB)}(E,e)$ from~$\rmE_\infty^{s,t}$ with~$t-s=1$. Since~$\rmd_2$ is the last differential that may be non-zero, this is only a two-stage process.

The non-zero part of the~$\rmE_1$ page of the spectral sequence is not very populated. 
There are only five terms~$\rmE^{s,t}_1$ that can be non-zero, and only three~$\rmd_1$ differentials between them. The situation can be illustrated as follows.
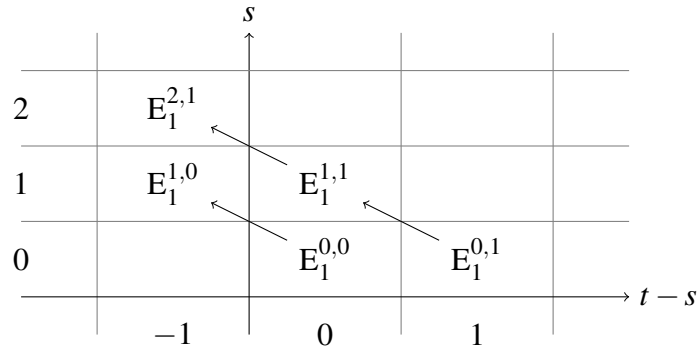
\begin{center}
\begin{tikzpicture}[xscale=2]

\draw[gray] (-1, -0.5) -- (-1, 3.5);
\draw[->] (0, -0.5) -- (0, 3.5) node[above] {$s$};
\draw[gray] (1, -0.5) -- (1, 3.5);
\draw[gray] (2, -0.5) -- (2, 3.5);

\draw (-1.5, 0.5) node [fill=white] {$0$};
\draw (-1.5, 1.5) node [fill=white] {$1$};
\draw (-1.5, 2.5) node [fill=white] {$2$};

\draw[->] (-1.5, 0) -- (2.5, 0) node[right] {$t-s$};
\draw[gray] (-1.5, 1) -- (2.5, 1);
\draw[gray] (-1.5, 2) -- (2.5, 2);
\draw[gray] (-1.5, 3) -- (2.5, 3);

\draw (-0.5, -0.5) node [fill=white] {$-1$};
\draw (0.5, -0.5) node [fill=white] {$0$};
\draw (1.5, -0.5) node [fill=white] {$1$};

\draw (0.5, 0.5) node [fill=white] {$\rmE^{0,0}_1$};
\draw (-0.5, 1.5) node [fill=white] {$\rmE^{1,0}_1$};
\draw (1.5, 0.5) node [fill=white] {$\rmE^{0,1}_1$};
\draw (0.5, 1.5) node [fill=white] {$\rmE^{1,1}_1$};
\draw (-0.5, 2.5) node [fill=white] {$\rmE^{2,1}_1$};

\draw [->] (0.25, 0.75) -- (-0.25, 1.25);
\draw [->] (1.25, 0.75) -- (0.75, 1.25);
\draw [->] (0.25, 1.75) -- (-0.25, 2.25);

\end{tikzpicture}
\end{center}  

However, we note in advance that the terms will not necessarily carry the structure of abelian groups, as one might be used to in spectral sequences. We will now describe these terms and the differentials between them, so as to obtain a description of the~$\rmE_2$ page.


\subsection{The terms with~\texorpdfstring{$t=0$}{t=0}}

Let us first look at the two terms with~$t=0$. 

\begin{definition}
We define
\begin{equation}
\rmE_1^{0,0}=\prod_{x\in\bbB}\Iso\bbB(x,x),
\end{equation}
which is a monoid, and
\begin{equation}\label{eq:new_monoid}
\rmE_1^{1,0}=\prod_{y,z\in\bbB}\Iso\bfFun(\bbB(y,z),\bbB(y,z)),
\end{equation}
which is a monoid as well. 
\end{definition}

\begin{remark}\label{rem:E2confusion}
Here is a common source of confusion: Each functor~\hbox{$F\colon\bbB(y,z)\to\bbB(y,z)$} induces a map~$\Iso(\bbB(y,z))\to\Iso(\bbB(y,z))$, and this map only depends on the isomorphism class of the functor~$F$.
This gives us a (tautological) homomorphism
\[
\tau(y,z)\colon\Iso\bfFun(\bbB(y,z),\bbB(y,z))\longrightarrow\Map(\Iso\bbB(y,z),\Iso\bbB(y,z))
\]
of monoids, but this homomorphism is neither injective nor surjective in general. It is the source that it relevant in~\eqref{eq:new_monoid}, not the less useful target.
\end{remark}

There are two natural homomorphisms~{$\rmd_1',\rmd_1'':\rmE_1^{0,0}\to\rmE_1^{1,0}$} of monoids, given by sending a family~$([\,P(x)\,]\,|\,x\in\bbB)$ of isomorphism classes of objects to either the equivalence class~\hbox{$[\,M\mapsto P(z)\otimes M\,]$} 
of the endo-functor~\hbox{$M\mapsto P(z)\otimes M$} or to the equivalence class of the endo-functor functor~\hbox{$[\,M\mapsto M\otimes P(y)\,]$} respectively. 
The differential~\hbox{$\rmd_1\colon\rmE_1^{0,0}\to\rmE_1^{1,0}$} should be though of as the difference of them:

\begin{definition}\label{def:E200}
The monoid~$\rmE_2^{0,0}$ is defined as the equalizer of~$\rmd_1'$ and~$\rmd_1''$.
\end{definition}

\begin{proposition}\label{prop:E200<=ZHo}
There are injections
\[
\rmE_2^{0,0}\leqslant\rmZ(\bfHo(\bbB))\leqslant\rmE_1^{0,0}
\]
of monoids.
\end{proposition}

\begin{proof}
On the one hand, an element in~$\rmE_2^{0,0}$ is an element~$\rmE_1^{0,0}$ that lies in the equalizer of~$\rmd_1'$ and~$\rmd_1''$. 
These are the families~$([\,P(x)\,]\,|\,x\in\bbB)$ of isomorphism classes of objects such that the two endo-functors~$M\mapsto P(z)\otimes M$ and~\hbox{$M\mapsto P(z)\otimes M$}
are {\em naturally} isomorphic. 
On the other hand, an element in the center~$\rmZ(\bfHo(\bbB))$ of the classifying category is just a family~$([\,P(x)\,]\,|\,x\in\bbB)$ such that the objects~$P(z)\otimes M$ 
and~\hbox{$P(z)\otimes M$} are isomorphic~(perhaps not naturally) for all~$M$. 

Finally, notice that the center~$\rmZ(\bfHo(\bbB))$ of the classifying category can be considered as the equalizer of the maps~$\prod_{y,z}\tau(y,z)\rmd_1'$ and~$\prod_{y,z}\tau(y,z)\rmd_1''$,
where the maps~$\tau(y,z)$ are the tautological maps defined in Remark~\ref{rem:E2confusion}.
\end{proof}


\subsection{The terms with~\texorpdfstring{~$t=1$}{t=1}}

Let us now look at the three terms with~$t=1$. 

\begin{definition}
We define  
\begin{equation}
\rmE_1^{0,1}=\prod_{x\in\bbB}\Aut_{\bbB(x,x)}(\Id(x))
\end{equation}
and 
\begin{equation}
\rmE_1^{1,1}=\prod_{y,z\in\bbB}\Aut_{\bfEnd(\bbB(y,z))}(\id),
\end{equation}
which are both abelian groups.
\end{definition}

Again, there are two distinguished homomorphisms~$\rmd_1',\rmd_1''\colon\rmE_1^{0,1}\to \rmE_1^{1,1}$, this time of abelian groups. One can be described as sending a family
\[
u=(u(x)\colon\Id(x)\to\Id(x)\,|\,x\in\bbB)
\]
of automorphisms to the natural transformation~$(u(z)\otimes\id(M)\,|\,M)$ and the other sends it to~$(\id(M)\otimes u(y)\,|\,M)$. 
Actually the targets are slightly different, but the tensor structure can be used to compare the two, using the diagram
\[
\xymatrix{
\Id(z)\otimes M\ar[d]_{u(z)\otimes\id(M)}\ar@{<->}[r]^-{\cong} & M\ar[d]\ar@{<->}[r]^-{\cong} & M\otimes\Id(y)\ar[d]^{\id(M)\otimes u(y)}\\
\Id(z)\otimes M\ar@{<->}[r]_-{\cong} & M\ar@{<->}[r]_-{\cong} & M\otimes\Id(y).
}
\]
The differential~$\rmd_1\colon\rmE_1^{0,1}\to\rmE_1^{1,1}$ is the difference of~$\rmd_1'$ and~$\rmd_1''$.

\begin{definition}
We define~$\rmE_2^{0,1}$ to be the equalizer of the two homomorphisms~$\rmd_1'$ and~$\rmd_1''$, this is the kernel of the difference~$\rmd_1=\rmd_1'-\rmd_1''$.
\end{definition}

Direct inspection gives the following result.

\begin{proposition}\label{prop:E201=Aut}
There is an isomorphism
\[
\rmE_2^{0,1}\cong\Aut_{\bfZ(\bbB)}(E,e)
\]
of abelian groups.
\end{proposition}

This already finishes the calculation of one of the basic invariants of~$\bfZ(\bbB)$ as a braided tensor category, the (abelian) automorphism group of its tensor unit. 

\subsection{Measuring the failure of injectivity}

We will now proceed to calculate the (abelian) monoid~$\Iso(\bfZ(\bbB))$ of isomorphism classes of objects as well. In order to do so, we need to describe the remaining group on the~$\rmE_1$ page. 

Given three objects~$x,y,z$ of~$\bbB$, it will be useful to write
\[
\bfF(x,y,z)=\bfFun(\bbB(y,z)\times\bbB(x,y),\bbB(x,z))
\]
as an abbreviation for the functor category.

\begin{definition}\label{def:E121}
For any family~$P=(P(x)\,|\,x\in\bbB)$ we define
\begin{equation}
\rmE_1^{2,1}(P)=\prod_{x,y,z\in\bbB}\Aut_{\bfF(x,y,z)}(P\otimes?\otimes??),
\end{equation}
where~$P\otimes?\otimes??:\bbB(y,z)\times\bbB(x,y)\to\bbB(x,z)$ denotes the functor
\[
(M, N)\longmapsto P(z)\otimes M\otimes N.
\]
This is a group that may not be abelian. 
\end{definition}

It is clear that an isomorphism~\hbox{$P\cong Q$} determines an isomorphism~$\rmE_1^{2,1}(P)\cong\rmE_1^{2,1}(Q)$ of groups. For~$P=E$, the group
\[
\rmE_1^{2,1}(E)\cong \prod_{x,y,z\in\bbB}\Aut_{\bfF(x,y,z)}(\ot)
\] 
receives three homomorphisms from the abelian group~$\rmE_1^{1,1}$: One sends a family 
\[
f=(f(M)\colon M\longrightarrow M\,|\,y,z\in\bbB,M\in\bbB(y,z))\in\rmE_1^{1,1}
\]
of natural automorphisms of the identity to the family
\[
(f(M)\otimes \id(N)\,|\,x,y,z\in\bbB,M\in\bbB(y,z),N\in\bbB(x,y)),
\]
and the other ones are given similarly by
\[
(\id(M)\otimes f(N)\,|\,x,y,z\in\bbB,M\in\bbB(y,z),N\in\bbB(x,y))
\] 
and
\[
(f(M\otimes N)\,|\,x,y,z\in\bbB,M\in\bbB(y,z),N\in\bbB(x,y)),
\]
respectively.

Lead by Equation~\eqref{eq:triangle}, we define a subgroup of the abelian group~$\rmE_1^{1,1}$ by
what should be thought of as the alternating sum of these three homomorphisms:

\begin{definition}
\begin{equation}
\rmZ_1^{1,1}=\{f\in\rmE_1^{1,1}\,|\,f(M\otimes N)=(\id(M)\otimes f(N))(f(M)\otimes\id(N))\}
\end{equation}
\end{definition}

Note that
\[
(\id(M)\otimes f(N))(f(M)\otimes\id(N))=f(M)\otimes f(N),
\] 
so that we can rewrite this definition as
\begin{equation}\label{eq:rewrite}
\rmZ_1^{1,1}=\{f\in\rmE_1^{1,1}\,|\,f(M\otimes N)=f(M)\otimes f(N)\}.
\end{equation}

\begin{definition}
We define
\begin{equation}
\rmB_1^{1,1}=\{(u\otimes\id(N))(\id(M)\otimes u^{-1})\,|\,u\in \rmE_1^{0,1}\}
\end{equation}
to be the image of the differential~$\rmd_1\colon\rmE_1^{0,1}\to\rmE_1^{1,1}$. 
\end{definition}

Recall that both of the groups~$\rmE_1^{0,1}$ and~$\rmE_1^{1,1}$ are abelian, so that taking the difference makes sense, and the image is a subgroup. It is then clear that~$\rmB_1^{1,1}\leqslant\rmZ_1^{1,1}$

\begin{definition}
We define
\begin{equation}
\rmE_2^{1,1}=\rmZ_1^{1,1}/\rmB_1^{1,1},
\end{equation}
which is also an abelian group.
\end{definition}

\begin{proposition}\label{prop:E211=kernel}
There is an isomorphism
\[
\rmE_2^{1,1}\cong\Ker(\Iso(\bfZ(\bbB))\to\rmZ(\bfHo(\bbB)))
\]
of abelian groups.
\end{proposition}

\begin{remark}
Notice that~$\Iso(\bfZ(\bbB))$ and~$\rmZ(\bfHo(\bbB))$ are abelian monoids, and are not necessarily groups.
However, each element from~$\Iso(\bfZ(\bbB))$ which maps into the identity element in~$\rmZ(\bfHo(\bbB))$ is invertible.
\end{remark}

\begin{proof}
The kernel above contains all isomorphism classes of objects~$(\,P(x)\,|\,x\in \bbB\,)$ in the Drinfeld center~$\bfZ(\bbB)$ such that there exists an isomorphism~$P(x)\cong\Id(x)$ for every object~$x$ of~$\bbB$.
We define
\begin{align*}
\rmZ_1^{1,1}&\longrightarrow \Iso(\bfZ(\bbB))\\
f&\longmapsto(P_f,p_f)
\end{align*}
as follows:
Every~$f\in\rmZ_1^{1,1}$ is a family of natural isomorphisms from the identity of~${\bbB}(y,z)$ to itself that is compatible with the tensor product as in~\eqref{eq:rewrite}.
We define the image~$(P_f,p_f)$ of~$f$ to be the following central object: For every object~$x\in\bbB$, we choose~$P_f(x)=\Id(x)$ to be the identity, and for every~\hbox{$M\in\bbB(x,y)$}, the isomorphism~\hbox{$p_f\colon M\cong M\ot P_f(x)\rightarrow P_f(y)\ot M\cong M$} is given by~$f(M)\colon M\rightarrow M$. The fact that the family~$f$ is compatible with the tensor products ensures that this construction produces a central object. 

The image~$(P_f,p_f)$ lies in the kernel under consideration, and in fact, every element in the kernel is of this form. If the object~$(P_f,p_f)$ is isomorphic to the identity object, then there is a family of isomorphisms~\hbox{$u(x)\colon P_f(x)=\Id(x)\rightarrow\Id(x)$} 
such that~$p_f$ is given by~\hbox{$(u(y)\ot\id(M))(\id(M)\ot u(x)^{-1})$}, where we have used the identifications~\hbox{$\Id(y)\ot M\cong M\cong M\ot\Id(x)$} again. 
This just means that~$f$ lies in~$\rmB_1^{1,1}$.
\end{proof}

The preceding proposition explains the potential failure of the injectivity of the characteristic homomorphism~\eqref{eq:characteristic}: 
If an element~$\rmZ(\bfHo(\bbB))$ can be lifted to an element in~$\bfZ(\bbB)$, then the abelian group~$\rmE_2^{1,1}$ acts on the different representatives in~$\Iso(\bfZ(\bbB))$. 
However, in monoids, this action need neither be free nor transitive.


\subsection{The~\texorpdfstring{$\rmE_2$}{E2} page}

The part of the~$\rmE_2$ page of the spectral sequence that can be non-trivial looks as follows.
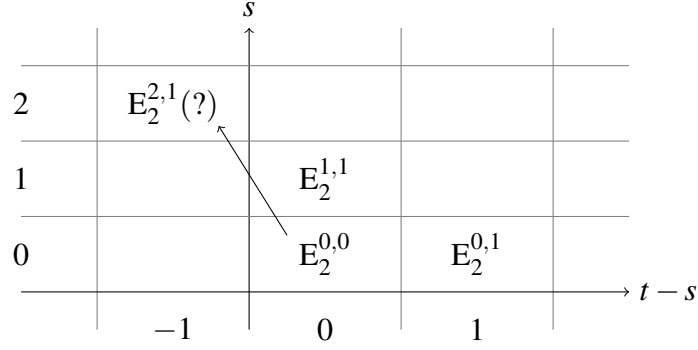
\begin{center}
\begin{tikzpicture}[xscale=2]

\draw[gray] (-1, -0.5) -- (-1, 3.5);
\draw[->] (0, -0.5) -- (0, 3.5) node[above] {$s$};
\draw[gray] (1, -0.5) -- (1, 3.5);
\draw[gray] (2, -0.5) -- (2, 3.5);

\draw (-1.5, 0.5) node [fill=white] {$0$};
\draw (-1.5, 1.5) node [fill=white] {$1$};
\draw (-1.5, 2.5) node [fill=white] {$2$};

\draw[->] (-1.5, 0) -- (2.5, 0) node[right] {$t-s$};
\draw[gray] (-1.5, 1) -- (2.5, 1);
\draw[gray] (-1.5, 2) -- (2.5, 2);
\draw[gray] (-1.5, 3) -- (2.5, 3);

\draw (-0.5, -0.5) node [fill=white] {$-1$};
\draw (0.5, -0.5) node [fill=white] {$0$};
\draw (1.5, -0.5) node [fill=white] {$1$};

\draw (0.5, 0.5) node [fill=white] {$\rmE^{0,0}_2$};
\draw (1.5, 0.5) node [fill=white] {$\rmE^{0,1}_2$};
\draw (0.5, 1.5) node [fill=white] {$\rmE^{1,1}_2$};
\draw (-0.5, 2.5) node [fill=white] {$\rmE^{2,1}_2(?)$};

\draw [->] (0.25, 0.75) -- (-0.2, 2.2);

\end{tikzpicture}
\end{center}
The three terms that have already been calculated are
\begin{align*}  
\rmE_2^{0,1}&\cong\Aut_{\bfZ(\bbB)}(E,e)
\end{align*}  
in the column~$t-s=1$ and
\begin{align*}  
\rmE_2^{0,0}&\leqslant\rmZ(\bfHo(\bbB))\\
\rmE_2^{1,1}&\cong\Ker(\Iso(\bfZ(\bbB))\to\rmZ(\bfHo(\bbB)))
\end{align*}
in the column~$t-s=0$, so that there is an exact sequence 
\[
0\longrightarrow
\rmE_2^{1,1}\longrightarrow
\Iso(\bfZ(\bbB))\longrightarrow
\rmZ(\bfHo(\bbB))
\]
that describes~$\Iso(\bfZ(\bbB))$, except for the image of the characteristic homomorphism. 

One may guess that the image of the characteristic homomorphism would be the kernel of a differential
\begin{equation}\label{eq:d2}
\rmd_2\colon\rmE_2^{0,0}\longrightarrow\rmE_2^{2,1},
\end{equation}
but the situation is more complicated: The question mark in~$\rmE^{2,1}_2(?)$ in the figure indicates that we do not have a single group~$\rmE_2^{2,1}$ as the target of a differential~\eqref{eq:d2}, but rather an entire family~$\rmE_2^{2,1}(P)$ of groups, one as the target of a tailored differential that acts on~$[\,P\,]\in\rmE_2^{0,0}$. This will now be explained in detail.


\subsection{Obstructions to surjectivity}\label{sec:obstructions}

We now address the following question: When can an element in the center~$\rmZ(\bfHo(\bbB))$ of the classifying category be lifted to an element in~$\Iso(\bfZ(\bbB))$ and hence in~$\bfZ(\bbB)$? 
Our proof of Proposition~\ref{prop:E200<=ZHo} already gives one condition: 
That element should lie in the submonoid~$\rmE_2^{0,0}$ that is cut out as the `kernel of~$\rmd_1$.' 
We may therefore right away start with an element in the monoid~$\rmE_2^{0,0}$.

Recall from our Definition~\ref{def:E200} that an element in the monoid~$\rmE_2^{0,0}$ is given by a family~$(\,P(x)\,|\,x\in\bbB\,)$ of objects for which the two functors~\hbox{$\rmR_{P(x)},\rmL_{P(y)}\colon\bbB(x,y)\to \bbB(x,y)$} given by~$M\mapsto M\ot P(x)$ and~$M\mapsto P(y)\ot M$, respectively, are naturally isomorphic. The family can be lifted to an object in the Drinfeld center~$\bfZ(\bbB)$ if and only if we can choose a family  
$p=(\,p_{x,y}\colon\rmR_{P(x)}\to\rmL_{P(y)}\,|\,x,y\in\bbB\,)$ of (natural) isomorphisms of functors such that we have a natural isomorphism 
\[
p_{x,z}(M\ot N) = (\id(M)\ot p_{x,y}(N))(p_{y,z}(M)\ot\id(N))
\]
for every pair of objects~$M\in \bbB(y,z)$ and~\hbox{$N\in\bbB(x,y)$}. 
To measure the failure of a given~$p$ to comply to these needs, we may consider the composition
\begin{equation}\label{eq:aut}
\rmd_2(p)_{x,y,z}(M,N)=(\id(M)\ot p_{x,y}(N))(p_{y,z}(M)\ot\id(N))p_{x,z}^{-1}(M\ot N),
\end{equation}
which is an automorphism of the functor~$(M,N)\mapsto P(z)\ot M\ot N$. If we let~$x$,~$y$, and~$z$ vary, then the family of the~$\rmd_2(p)_{x,y,z}$ defines an element~$\rmd_2(p)$ of the group~$\rmE_1^{2,1}(P)$ of our Definition~\ref{def:E121}. The automorphisms~\eqref{eq:aut} is the identity automorphism if and only if~$(P,p)$ lies in the Drinfeld center~$\bfZ(\bbB)$. This leads us to regard the automorphisms~\eqref{eq:aut} as the obstructions for~$(P,p)$ to be an object of the Drinfeld center. 

Of course, it is possible that~$(P,p)$ will not be an object of the Drinfeld center, but~$(P,p')$ will be, for some other family~$p'=(\,p'_{x,y}\,)$ of isomorphisms. Since the group of automorphisms of the functor~$(M,N)\mapsto P(z)\ot M\ot N$ is not abelian in general, we proceed as follows.

\begin{definition}
We define the {\em obstruction differential}$~\rmd_2$ on~$P$ by
\[
\rmd_2(P)=\{\,\rmd_2(p)\,|\,p=(\,p_{x,y}\colon\rmR_{P(x)}\to\rmL_{P(y)}\,|\,x,y\in\bbB\,)\,\}\subseteq\rmE_1^{2,1}(P),
\] 
where~$p$ runs over all possible isomorphisms of functors.
We will say that~$\rmd_2(P)$ {\em vanishes} on~$P$ if it contains the identity automorphism.
\end{definition}

The discussion preceding the definition proves the following result.

\begin{proposition}
The obstruction differential~$\rmd_2$ vanishes on~$P$ if for some choice of~$p$ the object~$(P,p)$ lies in the Drinfeld center~$\bfZ(\bbB)$.
\end{proposition}


\section{Groups and the category of bands}\label{sec:groups}

In this section, we present a detailed discussion of the various notions of centers, and in particular the Drinfeld center, 
in a situation that is genuinely different from tensor categories: categories and bicategories where the objects are (discrete) groups. 

A size limitation needs to be chosen, and we can and will assume that all groups under consideration are finite. Therefore, let~$\bfG$ denote the category of all finite groups. Definition~\ref{def:Drinfeld_center} requires a small category, so that it will be clear that the result is a set. However, our calculations will reveal that the result is a set anyway. We could also choose to work with a skeleton, and then note that the choice of skeleton does not affect the calculation, since any two skeleta are equivalent.

The category~$\bfG$ is the underlying category of a~$2$-category~$\bbG$, where~$\bfMor_\bbG(G,H)$ is the groupoid of homomorphisms~$G\to H$, 
which are the~$1$-morphisms of~$\bbG$, and the~$2$-morphisms~\hbox{$h\colon\alpha\to\beta$} between two homomorphisms~$\alpha,\beta\colon G\to H$ are the elements~$h$ in~$H$ that conjugate one into the other:
\[
\{\,h\in H\,|\,h\alpha(g)h^{-1}=\beta(g)\text{ for all }g\in G\,\}.
\] 

\begin{remark}
The classifying category~$\bfB=\bfHo(\bbG)$ is sometimes called the category of {\em bands} in accordance with its use in non-abelian cohomology and the theory of gerbes, see~\cite[IV.1]{Giraud}. 
It is customary to denote the conjugacy classes of homomorphisms~$G\to H$ by
\[
\Rep(G,H)=\Iso\bfMor_\bbG(G,H)=\rmMor_\bfB(G,H).
\]
These are the sets of morphisms in the classifying category~$\bfB=\bfHo(\bbG)$. Note that Giraud used the notation~$\Hex(G,H)$ instead of our~$\Rep(G,H)$.
\end{remark} 

The automorphism group of any given homomorphism~$\alpha\colon G \to H$ in the category~\hbox{$\bbG(G,H)=\bfMor_\bbG(G,H)$} of homomorphisms is the centralizer
\begin{equation}\label{eq:Zalpha}
\Aut_{\bbG(G,H)}(\alpha)=\{\,h\in H\,|\,h\alpha(g)h^{-1}=\alpha(g)\text{ for all }g\in G\,\}=\rmC(\alpha) 
\end{equation}
of the image of~$\alpha$. This is a subgroup of the group~$H$. We remark that the centralizers need not be abelian. For example, the centralizer of the constant homomorphism~\hbox{$G\to G$} is the entire group~$G$, 
whereas the center of~$G$ reappears as the centralizer of the identity~$G\to G$. These observations will be useful when we will determine the Drinfeld center of~$\bbG$. 

\subsection{Some ordinary centers}

Before we turn our attention towards the Drinfeld center, let us first describe the centers of the ordinary categories~$\bfG$ and~$\bfB=\bfHo(\bbG)$. 

\begin{proposition}\label{prop:ordinary_groups}
The centers of the categories~$\bfG$ and~$\bfB=\bfHo(\bbG)$ are isomorphic to the abelian monoid~$\{0,1\}$ under multiplication.
\end{proposition}

\begin{proof}
This is straightforward for the category~$\bfG$ of groups and homomorphisms. An element in the center thereof is a family~$P=(P_G\colon G\to G)$ of homomorphisms that are natural in~$G$. 
We can evaluate~$P$ on the full subcategory of cyclic groups, and since~$\bfMor(\bbZ/k,\bbZ/k)\cong\bbZ/k$, 
we see that~$P$ is determined by a profinite integer~$n$ in~$\widehat\bbZ$: we must have~$P_G(g)=g^n$ for all groups~$G$ and all of their elements~$g$. 
But, if~$n$ is not~$0$ or~$1$, then there are clearly groups for which that map is not a homomorphism. In fact, we can take symmetric or alternating groups, as we will see in the course of the rest of the proof.

Let us move on to the center of the classifying category~$\bfB=\bfHo(\bbG)$. Again, the homomorphisms~$g\mapsto g^0$ and~$g\mapsto g^1$ are in the center, 
and they still represent different elements, since they are not conjugate. In the classifying category, if~$[\,P\,]=([\,P_G\,]\colon G\to G)$ is an element in the center, 
testing against the cyclic groups only shows that there is a profinite integer~$n$ such that~$P_G(g)$ is {\em conjugate} to~$g^n$ for each group~$G$ and each of its elements~$g$. 
We will argue that no such family of homomorphisms~$P_G$ exists unless~$n$ is~$0$ or~$1$.

Let us call an endomorphisms~$\alpha\colon G\to G$ on some group~$G$ {\em of conjugacy type~$n$} if~$\alpha(g)$ is conjugate to~$g^n$ for all~$g$ in~$G$. 
We need to show that for all~$n$ different from~$0$ and~$1$ there exists at least one group~$G$ that does not admit an endomorphism of conjugacy type~$n$.

If~$|n|\geqslant2$, then we choose~$m\geqslant\max\{n,5\}$ and consider the subgroup~$G$ of the symmetric group~$\rmS(m)$ generated by the elements of order~$n$. 
Since the set of generators is invariant under conjugation, this subgroup is normal, and it follows that~$G=\rmA(m)$~(the  subgroup of alternating permutations) or~$G=\rmS(m)$. 
An endomorphisms~$\alpha\colon G\to G$ of conjugacy type~$n$ would have to be trivial because it vanishes on the generators. But then~$g^n$ would be trivial for all elements~$g$ in~$\rmA(m)\leqslant G$, 
which is absurd.

If~$n=-1$, then we first note that an endomorphisms~$\alpha\colon G\to G$ of conjugacy type~$-1$ is automatically injective. Hence, if~$G$ is finite, then it is an automorphism. 
Therefore we choose a nontrivial finite group~$G$ of odd order such that its outermorphism group is trivial.
~(Such groups exist, see~\cite{Horosevskii},~\cite{Dark} or~\cite{Heineken} for examples that also have trivial centers.) 
If~$\alpha\colon G\to G$ were an endomorphisms of conjugacy type~$-1$, then this would be an inner automorphism by the assumption on~$G$. 
Then~$\id\colon G\to G$, which represents the same class, would also be an endomorphisms of conjugacy type~$-1$. 
In other words, every element~$g$ would be conjugate to its inverse~$g^{-1}$, a contradiction since the order of~$G$ is odd.
\end{proof}

\subsection{The Drinfeld center}

Now that we have evaluated the centers of the ordinary categories of groups and bands, 
we are ready to apply the obstruction theory and spectral sequence introduced in Section~\ref{sec:ss} 
in order to determine the Drinfeld center~$\bfZ(\bbG)$ of the~$2$-category~$\bbG$ of groups. The following result describes the two basic invariants of~$\bfZ(\bbG)$.

\begin{proposition}
The maps
\[
\rmZ(\bfG)\longrightarrow\Iso\bfZ(\bbG)\longrightarrow\rmZ(\bfB=\bfHo(\bbG))
\]
are both isomorphisms, and the automorphism group of the identity object in~$\bfZ(\bbG)$ is trivial.
\end{proposition}

\begin{proof}
Let us first deal with the one entry in the spectral sequence that has~$t=0$. We already know that there is an upper bound~\hbox{$\rmE^{0,0}_2\leqslant\rmZ(\bfB=\bfHo(\bbG))$} 
by Proposition~\ref{prop:E200<=ZHo}. The center of the classifying category has been determined in the preceding Proposition~\ref{prop:ordinary_groups}. 
That result also makes it clear that all elements in the center of the classifying category lift to the Drinfeld center of~$\bbG$; 
they even lift to the center of the underlying category~$\bfG$. We deduce that the obstructions vanish.

Let us now deal with the two entries in the spectral sequence that have~$t=1$ and that determine the kernel of the map~$\Iso\bfZ(\bbG)\to\rmZ(\bfB=\bfHo(\bbG))$ 
and the automorphism group of the identity object in~$\bfZ(\bbG)$. We have
\[
\rmE_1^{0,1}=\prod_F\Aut_{\bbG(F,F)}(\id)\cong\prod_F\rmZ(F)
\]
by~\eqref{eq:Zalpha}, and we record that this is an abelian group. 

In order to determine the entry~$\rmE_1^{1,1}$, we start with the definition:
\[
\rmE_1^{1,1}=\prod_{G,H}\Aut_{\bfEnd(\bbG(G,H))}(\id).
\]
We know that the category~$\bbG(G,H)=\bfMor_\bbG(G,H)$ is a groupoid, and as such it is equivalent to the sum of the groups~$\rmC(\alpha)$, where~$\alpha$ runs through a system of representatives of~$\Rep(G,H)$ in~$\bfMor_{\bbG}(G,H)$. Since we are considering automorphisms of the identity object, we get
\[
\Aut_{\bfEnd(\bbG(G,H))}(\id)
\cong\prod_{[\,\alpha\,]\in\Rep(G,H)}\rmZ\rmC(\alpha),
\]
the product of the centers of the centralizers. We note again that this is an abelian group. This leaves us with 
\[
\rmE_1^{1,1}\cong\prod_{G,H}\;\prod_{[\,\alpha\,]\in\Rep(G,H)}\rmZ\rmC(\alpha).
\]

The differential~$\rmd_1\colon\rmE_1^{0,1}\to\rmE_1^{1,1}$ is the difference of the two coface homomorphisms. 
Therefore, up to an irrelevant sign, it is given on a family~$P=(P(F)\in\rmZ(F)\,|\,F)$ by
\begin{equation}\label{eq:diff1}
(\rmd_1 P)(\alpha)=P(H)-\alpha(P(G))\in\rmZ\rmC(\alpha)
\end{equation}
in the factor of~$\alpha\colon G\to H$. It follows that the~$\rmE_2^{0,1}$ entry in the spectral sequence consists of those families~$P$ such that~$P(H)=\alpha(P(G))$ for all~$G$,~$H$, 
and~$\alpha\colon G\to H$. Taking~$\alpha$ to be constant, we see that~$P(F)$ has to be trivial for all~$F$. This shows~$\rmE_2^{0,1}=0$. 
Therefore, by Proposition~\ref{prop:E201=Aut}, we deduce that~$\Aut_{\bfZ(\bbG)}(E,e)$ is indeed trivial. 

It remains to be shown that there are no more components than we already know.  Proposition~\ref{prop:E211=kernel} says that these are indexed by the group~$\rmE_2^{1,1}$. 
This group can be calculated as follows. Its elements are represented by elements in the subgroup~$\rmZ_1^{1,1}\leqslant\rmE_1^{1,1}$, 
the subgroup of elements~\hbox{$Q=(Q(\alpha)\in\rmZ\rmC(\alpha)\,|\,\alpha)$} such that
\begin{equation}\label{eq:cocycle}
Q(\gamma\beta)=Q(\gamma)+\gamma(Q(\beta)),
\end{equation}
again up to an irrelevant sign. We claim that each family with that property is already in the distinguished subgroup~$\rmB_1^{1,1}$, 
so that~$\rmZ_1^{1,1}=\rmB_1^{1,1}$ and~$\rmE_2^{1,1}=\rmZ_1^{1,1}/\rmB_1^{1,1}=0$.

That subgroup~$\rmB_1^{1,1}$ is the image of the differential~$\rmd_1$. Therefore, to prove the claim, let us be given a family~$Q=(Q(\alpha)\in\rmZ\rmC(\alpha)\,|\,\alpha)$ such that~\eqref{eq:cocycle} holds. 
We can then evaluate this family at the unique homomorphisms~$\alpha=\epsilon_F$ from the trivial group to~$F$, for each finite group~$F$, to obtain a family~$P(F)=Q(\epsilon_F)$, 
and that family is our candidate for an element~$P$ to hit the element~$Q$ under the differential~$\rmd_1$. And indeed, equation~\eqref{eq:cocycle} for~$\alpha=\gamma$ and~$\beta=\epsilon_G$ gives
\[
Q(\epsilon_H)=Q(\alpha \epsilon_G)=Q(\alpha)+\alpha(Q(\epsilon_G)).
\]
Rearranging this yields the identity
\[
Q(\alpha)=Q(\epsilon_H)-\alpha(Q(\epsilon_G))=P(H)-\alpha(P(G))=(\rmd_1 P)(\alpha),
\]
and this shows that~$Q$ is indeed in the image. We have proved the claim. 
\end{proof}

It seems reasonable to expect that similar arguments will determine the Drinfeld centers of related bicategories such as the ones coming from groupoids or topological spaces, etc. 
This will not be pursued further here. 
Instead, we will now turn our attention towards a class of examples that indicates the wealth of obstructions and nontrivial differentials that one can expect in general.


\section{Applications to fusion categories}\label{sec:fusion}

Fusion categories are tensor categories with particularly nice properties. They arise in many areas of mathematics and mathematical physics, such as operator algebras, conformal field theory, and Hopf algebras. A general theory of such categories has been developed in~\cite{ENO}. In this section, we use the theory developed so far in order to explain some constructions related to the centers of fusion categories that otherwise may seem to appear {\em ad hoc}. For completeness, let us start with the definition.

\begin{definition}
A {\em fusion category} is a semisimple abelian tensor category~$(\bfF,\otimes,I)$ over a field~$\frK$ of characteristic zero, usually assumed to be algebraically closed, 
with finitely many simple objects, such that~$\otimes$ is bilinear, and such that each object has a dual object.
\end{definition}

We are here interested in the bicategory~$\bbB(\bfF)$ with one object that is associated with such a fusion category~$\bfF$ as explained in Example~\ref{ex:tensor_categories}, and its Drinfeld center. Since fusion categories are special cases of tensor categories, this refers to the usual notion of a Drinfeld center of a tensor category, and as such it has been studied in other places. For example, the papers~\cite{Mueger},~\cite{Ostrik},~\cite{GNN09}, and~\cite{BV13} contain various results on the centers of fusion categories and related categories.

We will show that the spectral sequence introduced in Section~\ref{sec:ss} offers a systematic approach to the computation of the basic invariants of the Drinfeld center, by interpreting the different terms and differentials in the language of fusion categories.


\subsection{The first page}

We start by identifying the terms on the first page of the spectral sequence.

\begin{proposition}
If a fusion category~$\bfF$ has~$n$ isomorphism classes of simple objects, then the monoid~$\rmE^{0,0}_1$ can be identified as a set with~$\bbN^n$, the~$n$-fold product of the monoid~$\bbN$ of non-negative integers. The multiplication is given by the fusion coefficients. 
\end{proposition}

\begin{proof}
The monoid~$\rmE^{0,0}_1$ is the set of isomorphism classes of objects in our category. 
If the different simple objects of~$\bfF$ are~$X_1,\ldots,X_n$, then there is a canonical identification of~$\rmE^{0,0}_1$ with~$\bbN^{n}$, given by 
\[
(a_1,\ldots,a_n)\longleftrightarrow X_1^{\oplus a_1}\oplus\cdots\oplus X_n^{\oplus a_n}.
\] 
Notice that the direct sum on~$\bfF$ and the tensor product are two different operations.
\end{proof}

An element of the monoid~$\rmE^{1,0}_1$ is just an isomorphism class of endofunctors~$\bfF\to\bfF$. 

\begin{proposition}
If a fusion category~$\bfF$ over a field~$\frK$ has~$n$ isomorphism classes of simple objects, with representatives~$X_1,\dots,X_n$, then there are isomorphisms
\begin{align*}
\rmE_1^{0,1}&\cong \frK^\times\\
\rmE_1^{1,1}&\cong (\frK^\times)^n\\
\rmE_1^{2,1}&\cong \prod_{i,j}\Aut(X_i\otimes X_j).
\end{align*}
\end{proposition}

\begin{remark} 
In every fusion category with simple objects~$X_1,\dots,X_n$, the automorphism group of the object~\hbox{$X_1^{\oplus a_1}\oplus\cdots\oplus X_n^{\oplus a_n}$} 
is the group~$\prod_i\GL_{a_i}(\frK)$. All of the groups in the preceding proposition have this form.
\end{remark}

\begin{proof}
The group~$\rmE^{0,1}_1$ is the automorphism group of the tensor identity~$I$ of~$\bfF$. 
In a fusion category, the tensor identity is simple, and the endomorphism ring of each simple object is~$\frK$. Therefore this group is isomorphic to~$\frK^\times$.

The group~$\rmE^{1,1}_1$ is the automorphism group of the identity functor~$\id_\bfF\colon\bfF\to \bfF$. 
In the case of a fusion category, such an automorphism~$\alpha$ is specified by giving~$\alpha_i\colon X_i\to X_i$ for each~\hbox{$i=1,\ldots,n$}. 
In other words, such an automorphism is given by a set of~$n$ invertible scalars, and the group is isomorphic with~$(\frK^\times)^n$. 
Now, the two maps we have~$\rmE^{0,1}_1\to\rmE^{1,1}_1$ are the same. They are given by the diagonal embedding~$\frK^\times\to (\frK^\times)^n$. 

The group~$\rmE^{2,1}_1$ is the automorphism group of the tensor functor~$\otimes\colon \bfF\times\bfF\to\bfF$. 
Any such automorphism is given by a set of~$n^2$ invertible morphisms~$\beta_{i,j}\colon X_i\otimes X_j\to X_i\otimes X_j$. 
We can thus identify this group with the product~$\prod_{i,j}\Aut(X_i\otimes X_j)$.
\end{proof}


\subsection{The first differentials and the center of the classifying category}

The following result computes the first differential and the center of the classifying category.

\begin{proposition}
The abelian monoid~$\rmE^{0,0}_2$ is isomorphic to~$\rmZ(\bfHo(\bbB(\bfF)))$, the center of the classifying category.
\end{proposition}

\begin{proof}
The first differential~$\rmE^{0,0}_1\to\rmE^{1,0}_1$ is given by two maps, one which sends an object~$X$ to the functor~$\rmL_X\colon Y\mapsto X\otimes Y$ and the other one maps~$X$ to the~(isomorphism classes of the) functor~$\rmR_X\colon Y\mapsto Y\otimes X$. We are interested in the equalizer. 
These consist of all the~(isomorphism classes of) objects~$X$ for which there exists an isomorphism (of functors) 
between~$\rmL_X$ and~$\rmR_X$. For fusion categories, this is the same as the center of the classifying category. Indeed, since a fusion category~$\bfF$ is semi-simple, an additive functor~$\bfF\to\bfF$ is determined by its restriction to simple objects,
and the tensor product with a given object is an additive functor.
\end{proof}


\subsection{The universal grading group}

For each fusion category~$\bfF$, we define~$\bfF_{\mathrm{ad}}$ to be the fusion subcategory of~$\bfF$ consisting of all the direct summands of all objects of the form~$X_i\otimes X_i^*$. Notice that we take only the products of simple objects with their duals.
Then~$\bfF$ has a faithful grading by a group~$\rmU(\bfF)$, the {\em universal grading group of~$\bfF$} such that~$\bfF_{\mathrm{ad}}$ is exactly the trivially graded component. 
Moreover, each other faithful grading of~$\bfF$ is a quotient of this grading. See~\cite{GN08} for more details.

\begin{proposition}
There are isomorphisms
\begin{align*}
\rmE_2^{0,1}&\cong \frK^\times\\
\rmE_2^{1,1}&\cong\widehat{\rmU(\bfF)},
\end{align*}
where~$\widehat{\rmU(\bfF)}$ denotes the character group of the universal grading group of~$\bfF$.
\end{proposition}

\begin{proof}
The first claim follows from the fact that the differential from~$\rmE^{0,1}_1$ to~$\rmE^{1,1}_1$ is zero. 
This fact also implies that the group~$\rmB_1^{1,1}$ is the trivial group.

As for the second one, the three maps~$\rmE_1^{1,1}\to\rmE_1^{2,1}$ are the following. The first is given by sending a family~$(\alpha_i)$ to the family~$(\alpha_i)_{i,j}$. The second sends a family~$(\alpha_i)$ to the family~$(\alpha_j)_{i,j}$. The third is more complicated: It sends a family~$(\alpha_i)$ to~$(\beta_{i,j})_{i,j}$, where~$\beta_{i,j}$ acts by the scalar~$\alpha_k$ on the~$X_k$ component inside~$X_i\otimes X_j$. We can now offer two proofs, based on Proposition~\ref{prop:E211=kernel}.

First, we can identify~$Z^{1,1}_1$ and~$\widehat{\rmU(\bfF)}$: For each~$i=1,\ldots,n$, let~$g_i\in \rmU(\bfF)$ be the degree of~$X_i$ by the universal grading. 
Then the element in~$Z^{1,1}_1$ which corresponds to~$\phi\in\widehat{\rmU(\bfF)}$ is~$(\phi(g_i))$. We thus get an isomorphism~$\rmE^{1,1}_2\cong\widehat{\rmU(\bfF)}$. 

Second, we can also identify~$\widehat{\rmU(\bfF)}$ directly with the kernel of the characteristic homomorphism~$\Iso(\bfZ(\bbB(\bfF)))\to\rmZ(\bfHo(\bbB(\bfF)))$: 
Characters of the universal grading group~$\rmU(\bfF)$ are in one to one correspondence with objects of the Drinfeld center whose underlying object is the tensor unit. The central object corresponding to a character~$\phi$ is~$(I,\phi)$, the tensor unit~$I$, together with the isomorphism~$I\otimes X_i\to X_i\otimes I$ given by the scalar~$\phi(g_i)$. Here we identify both objects with~$X_i$ in the canonical way, and the endomorphism ring of~$X_i$ is~$\frK$.
\end{proof}

In conclusion, we have the following exact sequence:
\[
0\longrightarrow\widehat{\rmU(\bfF)}\longrightarrow\Iso(\bfZ(\bbB(\bfF)))\longrightarrow\rmZ(\bfHo(\bbB(\bfF)))
\]
for any fusion category~$\bfF$. It identifies the character group of the universal grading group as the measure of the failure of the center of the classifying category to detect information that is contained in the richer Drinfeld center.


\subsection{The obstruction differential}

The last part which we need to understand in the spectral sequence is the obstruction differential on the second page, that is~$\rmd_2\colon\rmE_2^{0,0}\to\rmE_2^{2,1}(?)$.
We have seen in Section~\ref{sec:obstructions} that the target~$\rmE_2^{2,1}(?)$ depends on the argument. Let us explain the obstruction that the differential encodes.

The monoid~$\rmE_2^{0,0}$ contains all objects~$X$ of~$\bfF$ for which the functors~$\rmL_X$ and~$\rmR_X$ are isomorphic. 
We will furnish a structure of a central object on~$X$ if we can find an isomorphism of functors~$\alpha\colon\rmL_X\to\rmR_X$ such that the following diagram will be commutative for every~$Y,Z\in\bfF$.
\begin{equation}\label{obst}
\xymatrix{ 
&X\otimes(Y\otimes Z) \ar[r]^{\alpha_{Y\otimes Z}}& (Y\otimes Z)\otimes X\ar[rd]\\
(X\otimes Y)\otimes Z\ar[rd]_{\alpha_Y\otimes\id_Z} \ar[ru] & & & Y\otimes(Z\otimes X)\\
 & (Y\otimes X)\otimes Z\ar[r] & Y\otimes (X\otimes Z) \ar[ur]_{\id_Y\otimes \alpha_Z} & }
\end{equation}
In other words, if for all choices of~$\alpha$, the identity is in the set of loops of diagram~\eqref{obst}.

Of course, different choices of~$\alpha$ might furnish different central structures on~$X$. If we take~$X$ to be the tensor unit, this obstruction indeed lies in~$\rmE^{2,1}_2(I)$.


\subsection{Vector spaces graded by groups}

Let~$G$ be a finite group, and assume that our fusion category~$\bfF$ is~$\bfVec_G^{\omega}$, that is, the category of~$G$-graded vector spaces, in which the associativity constraint is deformed by the class~\hbox{$[\,\omega\,]\in\rmH^3(G,\frK^\times)$} of a~$3$-cocycle~$\omega$, as in~\cite[Section~2]{ENO}. Briefly, the simple objects~$X_g$ in~$\bfVec_G^{\omega}$ are indexed by the elements~$g$ of the group~$G$. The tensor product is given by~$X_g\ot X_h=X_{gh}$ and the associativity constraint
\[
X_{ghk}=(X_g\ot X_h)\ot X_k\longrightarrow X_g\ot (X_h\ot X_k)=X_{ghk}
\]
is given by the scalar~$\omega(g,h,k)$. We can think of this category as the category of vector bundles over the set~$G$, with deformed associativity constraints.

The automorphism group of the neutral object in the Drinfeld center is given by~$\frK^\times$, as for every fusion category. 

The abelian monoid of isomorphism classes of objects is more interesting in this case: We can view~$\rmE_1^{0,0}$ as the multiplicative monoid underlying the group semi-ring~$\bbN G$, and then the elements of~$\rmE_2^{0,0}$ are exactly the central elements. The universal grading group~$\rmU(\bfVec_G^{\omega})$ is given by the group~$G$. This shows that the characteristic homomorphism is not injective as soon as the group~$G$ admits a non-trivial character.

These observations determine the~$\rmE_2$-page of the spectral sequence except for the obstruction differential.

Let us begin by understanding the obstruction differential~$\rmd_2$ in case we are taking a central element~$z\in G$, 
and ask if the corresponding simple object~$X_z$ has a structure of a central object. By inspecting the diagram~\eqref{obst} above, 
we get that~$z$ and~$\omega$ together furnish a~$2$-cocycle~$\gamma_{\omega,z}$ on~$G$, given by 
\[
\gamma_{\omega,z}(g,h)=\omega(g,h,z)\omega^{-1}(g,z,h)\omega(z,g,h).
\]
Assume that an object~$X_z$ has a structure of a central object. We can then choose isomorphisms~$\phi_g\colon X_g\ot X_z\rightarrow X_z\ot X_g$ for every element~$g$ in~$G$,
such that diagram~\eqref{obst} is commutative, with~$X=X_g$,~$Y=X_h$ and~$Z=X_z$.
We have~\hbox{$X_z\ot X_g=X_{zg}=X_{gz}=X_g\ot X_z$}.
This means that the isomorphism~$\phi_g$ can be given by a scalar.
By inspecting the diagram~\eqref{obst} again, we see that its commutativity boils down to the equation~$\del(\phi)=\gamma_{\omega,z}$.
In other words, the object~$X_z$ will have a central structure if and only if the class of~$\gamma_{\omega,z}$ vanishes in~$\rmH^2(G,\frK^{\times})$.
In that case, the different central structures on~$X_z$ are in one to one correspondence with characters of~$G$, as in the case when~$z$ is the neutral element of our group~$G$.

More generally, we can ask when the direct sum~$X_z^{\oplus m}$ of~$m$ copies of~$X_z$ will have a central structure. For this, we need the twisted group algebra~$\frK^{\gamma_{\omega,z}}G$. This~$\frK$-algebra has a~$\frK$-basis~$\{\,u_g\,|\,g\in G\,\}$ and the multiplication is given by
\[
u_gu_h=\gamma_{\omega,z}(g,h)u_{gh}.
\]
The~$2$-cocycle condition ensures that this algebra is associative.
Now, the automorphism group of the object~$X_z^{\oplus m}$ is isomorphic to the general linear group~$\GL_m(\frK)$.
The commutativity of the diagram~\eqref{obst} with~$Z=X_z^{\oplus m}$ will mean now that we can choose elements~$v_g\in\GL_m(\frK)$
such that~$v_gv_h=\gamma_{\omega,z}(g,h)v_{gh}$. In other words, the algebra~$\frK^{\gamma_{\omega,z}}G$ has an~$m$ dimensional representation,
and the different structures of central objects on~$X_z^{\oplus m}$ correspond to the different isomorphism classes of those representations of the twisted group algebra~$\frK^{\gamma_{\omega,z}}G$ that have dimension~$m$. Since~$\frK^{\gamma_{\omega,z}}G$ is finite-dimensional and semi-simple when~$\mathrm{char}(\frK)=0$, there are only finitely many such isomorphism classes.

The general obstructions can be understood now in a similar way: If~$g_1,\ldots,g_r$ are representatives of the conjugacy classes of~$G$, 
then we denote by~$Y_i$ the direct sum of all~$X_g$ such that~$g$ is conjugate to~$g_i$. We denote by~$y_i$ the corresponding sum in~$\bbN G$. 
Then any central element in~$\bbN G$ can be written as a sum~$\sum_i a_i y_i$ for~$a_i\in \bbN$. The possible central structures on the object~$\oplus_i Y_i^{\oplus a_i}$ are given by tuples~$([\,V_i\,])_i$ where~$[\,V_i\,]$ is an isomorphism class of a representation of~$\frK^{\gamma_{\omega,g_i}}\rmC_G(g_i)$ of dimension~$a_i$.

\begin{proposition}
For every group~$G$, class~$\omega\in\rmH^3(G,\frK^{\times})$, and central element~$z\in \rmZ(G)$ such that~$\gamma_{\omega,z}$ is non trivial, the isomorphism class of the object~$X_z$ is not in the image of the characteristic homomorphism.
\end{proposition}

\begin{example}
Let~$n\geqslant2$ be any integer, and set~$G=(\bbZ/n)^3$. We denote a~$\bbZ/n$-basis of the abelian group~$G$ by~$\{e_1,e_2,e_3\}$, and we write~$\zeta$ 
for a primitive~$n$-th root of unity in our field~$\frK$.
We consider the~$3$-cocycle~$\omega$ defined by
\[
\omega(e_1^{a_{11}}e_2^{a_{12}}e_3^{a_{13}},e_1^{a_{21}}e_2^{a_{22}}e_3^{a_{23}},e_1^{a_{31}}e_2^{a_{32}}e_3^{a_{33}}) = \zeta^{a_{11}a_{22}a_{33}}.
\]
in~$\rmH^3(G,\frK^{\times})$. It arises as the cup product of three~$\bbZ/n$-linearly independent elements in the group~$\rmH^1(G,\frK^{\times})\cong(\bbZ/n)^{\oplus3}$. Since the group~$G$ is abelian, every element is central; we choose~$z=e_1$. A direct calculation shows that 
\[
\gamma_{\omega,z}(e_1^{a_{11}}e_2^{a_{12}}e_3^{a_{13}},e_1^{a_{21}}e_2^{a_{22}}e_3^{a_{23}}) = \zeta^{a_{12}a_{23} + a_{11}a_{22}},
\]
which is non-trivial. Furthermore, the Wedderburn decomposition of the algebra~$\frK^{\gamma_{\omega,z}}G$ is given by
\[
\frK^{\gamma_{\omega,z}}G\cong\bigoplus_{i=1}^n\rmM_n(\frK)
\]
Thus, the class of the~$m$-fold direct sum~$X_z^{\oplus m}$ will be in the image of the characteristic homomorphism if and only if~$m$ is a multiple of~$n$.
\end{example}


\section{The bicategory of rings and bimodules}\label{sec:bimodules}

In this section, we discuss another important example of a bicategory, this time one that has several objects, and that is not a (strict)~$2$-category, 
the bicategory~$\bbM$ of rings and bimodules, see~\cite[2.5]{Benabou} and~\cite[XII.7]{MacLane}. The objects in~$\bbM$ are the (associative) rings~$A, B, C,\dots$~(with unit). The category~$\bbM(A,B)$ is the category of~$(B,A)$-bimodules~$M$ and their homomorphism. (As before in Section~\ref{sec:groups}, some size limitation on the rings and bimodules is needed to keep the categories and their centers relatively small. We will not further comment on this, since it is again inessential to our calculations.) Composition in~$\bbM$ is given by the tensor product of bimodules, and the~$(A,A)$-bimodule~$A$ is the identity object in the category~$\bbM(A,A)$. Every morphism~$f\colon A\to B$ in the~(ordinary) category of rings gives rise to a~$(B,A)$-bimodule~$B_f$, where~$A$ acts via~$f$. The identity object corresponds to the identity~$\id_A$ in the category~$\bbM(A,A)$. The~(ordinary) category of rings has the ring~$\bbZ$ as an initial object, and we use this observation repeatedly in this section. 

We can now begin to study our spectral sequence for the bicategory~$\bbM$. As in Section~\ref{sec:groups}, we start with the (ordinary) center of the classifying category.

\parbox{\linewidth}{\begin{proposition}\label{prop:ZHoM}
The center~$\rmZ(\bfHo(\bbM))$ of the classifying category~$\bfHo(\bbM)$ of the bicategory~$\bbM$ is the abelian monoid of isomorphism classes of abelian groups under the multiplication induced by the tensor product, with the isomorphism class of the infinite cyclic group as unit.
\end{proposition}}

\begin{proof}
Consider the monoid
\[
\rmE_1^{0,0}=\prod_{A\in\bbM} \Iso\bbM(A,A).
\]
An element in this monoid is given by a family of isomorphism classes of~$(A,A)$-bimodules~$(\,P(A)\,|\,A\in\bbM\,)$.
The monoid
\[
\rmE_1^{1,0}=\prod_{A,B\in\bbM} \Iso\bfFun (\bbM(A,B),\bbM(A,B))
\] 
contains families of isomorphism classes of endofunctorsof the category of~$(B,A)$-bimodules, one such for every two rings~$A$ and~$B$ in~$\bbM$.
The first differential vanishes on the isomorphism class of the family~$(\,P(A)\,|\,A\in\bbM\,)$ if and only if the two functors~$N\mapsto N\ot_A P(A)$ and~$N\mapsto P(B)\ot_B N$ are isomorphic on the category of~$(B,A)$-bimodules.

Consider in particular the case where~$B=\bbZ$ is the initial ring, and the~$(\bbZ,A)$-bimodule~$N$ is~$A$ itself. In this particular case, we see that
the two~$(\bbZ,A)$-bimodules~\hbox{$A\ot_A P(A)$} and~$P(\bbZ)\ot A$ are isomorphic to each other and hence to~$P(A)$. (Here and in the following we write~$\otimes=\otimes_\bbZ$ for readability.) We see that, up to isomorphism, the entire family is determined by the abelian group~$P(\bbZ)$. Conversely, given any abelian group~$M$, then the family of the~$P(A)=M\otimes A$ admits the structure of a central object. In other words, the submonoid~$\rmE_2^{0,0}$ consists of all isomorphism class of families of the form
\[
(\,P(A)\,|\,A\in\bbM\,)=(\,P(\bbZ)\otimes A\,|\,A\in\bbM\,),
\]
where~$P(\bbZ)$ is an arbitrary abelian group. We thus have~$\rmE_2^{0,0}=\rmZ(\bfHo(\bbM))$. It is clear from the definition of~$\bbM$ that the composition is given by the tensor product and the isomorphism class of the abelian group~$\bbZ$, which corresponds to the family of isomorphism class of the~$(A,A)$-bimodules~$A$, is the unit.
\end{proof}

We remark that we have seen in the proof that the usual symmetry will furnish a central structure on any family of bimodules of the form~$(\,P(A)=M\otimes A\,|\,A\in\bbM\,)$. In particular, this holds for the unit given by the family~$E=(\,E(A)=A\,)$ of~$(A,A)$-bimodules~$A$.

We can now direct our attention to the Drinfeld center of~$\bbM$ itself. We start by determining the automorphism group of its tensor unit.

\begin{proposition}
The automorphism group~$\Aut_{\bfZ(\bbM)}(E,e)$ of the tensor unit~$(E,e)$ in the Drinfeld center~$\bfZ(\bbM)$ of the bicategory~$\bbM$ is cyclic of order~$2$.
\end{proposition}

\begin{proof}
It is well known (and easy to prove) fact that the automorphism group of~$A$ as an~$(A,A)$-bimodule is canonically isomorphic to the group~$\rmZ(A)^{\times}$ of invertible elements in the center of the ring~$A$. So
\[
\rmE_1^{0,1}\cong\prod_{A\in\bbM}\rmZ(A)^{\times}
\]
follows. Similarly, any automorphism of the identity functor on the category of~$(B,A)$-bimodules, which is isomorphic to the category of~$B\ot A^\op$-modules, is determined by its action on the generator, so that it will be given as multiplication by an invertible element of the center of the ring~$B\ot A^\op$.
We thus have
\[
\rmE_1^{1,1}=\prod_{A,B\in\bbM}\rmZ(B\ot A^\op)^{\times}.
\]
Using these isomorphisms as identifications, the differential~$\rmd_1\colon\rmE_1^{0,1}\rightarrow\rmE_1^{1,1}$ then sends a family~$(\,z_A\,|\,A\in\bbM\,)$ to the family~$(\,z_B^{}\ot z_A^{-1}\,|\,A,B\in\bbM\,)$. It follows that, if the differential on the family~$(\,z_A\,|\,A\in\bbM\,)$ vanishes, then each element~$z_A$ must be the image in~$A$ of some invertible element in (the center of)~$\bbZ$. Because the only such elements in~$\bbZ$ are~$\pm 1$, we conclude that the only families in~$\rmE_2^{0,1}$ are~$(\,1_A\,|\,A\in\bbM\,)$ and~$(\,-1_A\,|\,A\in\bbM\,)$, and thus~$\Aut_{\bfZ(\bbM)}(E,e)\cong \rmE_2^{0,1}\cong\bbZ/2$.
\end{proof}

It remains for us to describe the isomorphism classes of objects in the center~$\bfZ(\bbM)$ of~$\bbM$. We will do this now by showing that the characteristic homomorphism is bijective, so that Proposition~\ref{prop:ZHoM} gives the desired result. 

\begin{proposition}
For the bicategory~$\bbM$ of rings and bimodules, the characteristic homomorphism~$\Iso(\bfZ(\bbM))\to\rmZ(\bfHo(\bbM))$ is an isomorphism. 
\end{proposition}

\begin{proof}
We have already explained the fact that an element in the center~$\rmZ(\bfHo(\bbM))$ of the classifying category~$\bfHo(\bbM)$ admits a structure of a central object that lives in~$\bfZ(\bbM)$. Therefore, all obstructions must vanish and the characteristic homomorphism is surjective.

The group~$\rmZ_1^{1,1}$ is the subgroup that consists of all families of invertible central elements~$z_{A,B}\in\rmZ(B\ot A^\op)^{\times}$
that satisfy the following condition: For every three rings~$A,B$ and~$C$, the element 
\[
(\id_C\ot m_B\ot\id_A)(z_{B,C}\ot z_{A,B})\in C\ot B\ot A^{\op}
\]
is equal to~$z_{A,B}\ot 1$, where we identify~$C\ot B\ot A^\op$ with~$C\ot A^\op\ot B$ using the symmetry, and where we denote the multiplication~$B\otimes B\to B$ by~$m_B$.

Similarly to before, by studying this equation for the special case~$B=\bbZ$ 
we learn that each such family must be of the form~$z_{A,B}=z_B^{}\ot z_A^{-1}$ for some invertible central elements~\hbox{$z_A\in\rmZ(A)^\times$}: Set~$z_A=z_{A,\bbZ}$ and note that~$z_{\bbZ,A}=z_A^{-1}$. 
In other words, a family in the group~$\rmZ_1^{1,1}$ will automatically be in the image of the differential~$\rmd_1$, which is the subgroup~$\rmB_1^{1,1}$. We conclude that the group~$\rmE_2^{1,1}$ is trivial, and therefore the characteristic homomorphism is also injective in this case.
\end{proof}

The following statement summarizes the results of this section. 

\begin{theorem}
The Drinfeld center~$\bfZ(\bbM)$ of the bicategory~$\bbM$ of rings and bimodules is isomorphic to the category of abelian groups with the monoidal structure given by the usual tensor product of abelian groups, and the unit given by the group of integers.
\end{theorem}

\begin{remark}
The obvious generalization from the bicategory of rings (that is~$\bbZ$-algebras) to the bicategory of~$R$-algebras~(over a given commutative ring~$R$) can be proven with the same techniques; this only requires a more ornamented notation. 
\end{remark}


\section*{Acknowledgment}

Both authors were supported by the Danish National Research Foundation through the Centre for Symmetry and Deformation (DNRF92).



\vfill

\parbox{\linewidth}{%
Ehud Meir\\
Department of Mathematical Sciences\\
University of Copenhagen\\
2100 Copenhagen \O\\
DENMARK\\
\href{mailto:meir@math.ku.dk}{meir@math.ku.dk}\\
\phantom{}\\
Markus Szymik\\
Department of Mathematical Sciences\\
NTNU Norwegian University of Science and Technology\\
7491 Trondheim\\
NORWAY\\
\href{mailto:markus.szymik@math.ntnu.no }{markus.szymik@math.ntnu.no }}

\end{document}